\documentclass[11pt]{article}

\usepackage[margin=1in]{geometry}
\usepackage{amsmath,amssymb,amsthm,mathtools}
\usepackage{microtype}
\usepackage{enumitem}
\usepackage[hidelinks]{hyperref}

\allowdisplaybreaks
\setlist[enumerate]{leftmargin=*,itemsep=2pt,topsep=4pt}
\setlist[itemize]{leftmargin=*,itemsep=2pt,topsep=4pt}

\newtheorem{theorem}{Theorem}[section]
\newtheorem{proposition}[theorem]{Proposition}

\newtheorem{corollary}[theorem]{Corollary}
\theoremstyle{definition}
\newtheorem{definition}[theorem]{Definition}
\newtheorem{problem}[theorem]{Problem}
\theoremstyle{remark}

\newcommand{\exlin}{\operatorname{ex}^{\mathrm{lin}}}
\newcommand{\clin}{c^{\mathrm{lin}}}
\newcommand{\PG}{\operatorname{PG}}
\newcommand{\AG}{\operatorname{AG}}
\newcommand{\R}{\mathbb{R}}
\newcommand{\rank}{\operatorname{rank}}
\newcommand{\calE}{\mathcal E}
\newcommand{\calL}{\mathcal L}

\title{Blocking Amalgamations, Maximal Arcs, and Generalized Crowns}
\author{Mahesh Ramani}
\date{16 August 2026}

\begin{document}
\maketitle

\begin{abstract}
Let \(C^r_{1,k}\) be the \(r\)-uniform \(k\)-crown and put
\(h=r-k+2\).  For a finite linear intersecting \(r\)-uniform hypergraph
\(G\), let \(\tau_h(G)\) be the minimum size of a set meeting every edge
of \(G\) in at least \(h\) vertices, and define
\[
 \rho_{r,k}=\sup_G\frac{|E(G)|}{\tau_h(G)}.
\]
We prove that every fixed pair \((G,B)\), with \(B\) an \(h\)-fold
transversal, yields
\[
 \operatorname{ex}^{\mathrm{lin}}_r(n,C^r_{1,k})
 \ge \frac{|E(G)|}{|B|}n-O_{G,B}(\sqrt n)
\]
for all sufficiently large \(n\).  Incidence counting gives
\(\rho_{r,k}\le r/h\), and equality is characterized after dualization by
a pairwise balanced design with a distinguished regular subfamily.

For \(r=q+1\), where \(q\) is a prime power, truncated projective planes
give
\[
 \frac qh\le \rho_{q+1,k}\le\frac{q+1}{h}.
\]
The upper endpoint is attained whenever a maximal \(h\)-arc exists; in
particular, if \(q\) is even and \(h\mid q\), then
\(\rho_{q+1,k}=(q+1)/h\).  Padding the truncated-plane construction gives
\[
 \rho_{r,r}=(1-o(1))\frac r2
\]
and, uniformly for each fixed \(\varepsilon>0\) and
\(\varepsilon r\le k\le r\),
\[
 \rho_{r,k}=(1+o(1))\frac{r}{r-k+2}.
\]
For nonintersecting templates, the corresponding transfer is governed by a
local safe-block condition that replaces the \(h\)-fold transversal requirement.
\end{abstract}

\section{Introduction}\label{sec:intro}

A hypergraph $H=(V(H),E(H))$ is \emph{$r$-uniform} if every edge has size $r$, and it is \emph{linear} if two distinct edges intersect in at most one vertex.  For $3\le k\le r$, the $r$-uniform $k$-crown $C^r_{1,k}$ consists of a base edge $e_0$ and petals $e_1,\ldots,e_k$ such that the petals are pairwise disjoint and
\[
 e_0\cap e_i=\{v_i\}\qquad(1\le i\le k),
\]
where $v_1,\ldots,v_k$ are distinct vertices of $e_0$.  The case $k=r$ is the \emph{full crown}.

For a fixed linear $r$-uniform hypergraph $F$, let
\[
 \exlin_r(n,F)
\]
denote the maximum number of edges in an $n$-vertex linear $r$-uniform $F$-free hypergraph, and define
\[
 \clin_r(F):=\limsup_{n\to\infty}\frac{\exlin_r(n,F)}{n}.
\]
The problem considered here is to construct dense crown-free hypergraphs and to determine how far a natural block-amalgamation method can be optimized.

Zhang, Broersma, and Wang introduced generalized crowns in higher uniformity and, when $r-1$ is a prime power, constructed full-crown-free examples with limiting density $r/(r-1)$ \cite{ZhangBroersmaWang2025}.  Adak later proved the general upper bound
\begin{equation}\label{eq:adak-upper-intro}
 |E(H)|\le \frac{(k-1)(r-1)+1}{r}\,|V(H)|
\end{equation}
for every linear $C^r_{1,k}$-free $r$-uniform hypergraph, after discarding a nonnegative correction term from his sharper statement \cite{Adak2026}.  At the full-crown endpoint $r=q+1$, this upper coefficient is asymptotic to $q$.

The construction in this paper separates two roles.  A finite \emph{block} $G$ supplies many edges.  A set of block vertices remains private in each copy and prevents a crown, while the remaining vertices are shared among many copies through an auxiliary incidence structure.  The private set is measured by a multiple transversal number.

For an integer $h\ge1$, an \emph{$h$-fold transversal} of a hypergraph $G$ is a set $B\subseteq V(G)$ satisfying
\[
 |B\cap E|\ge h\qquad(E\in E(G)).
\]
Its minimum size is denoted by $\tau_h(G)$.  For crowns, the relevant multiplicity is
\[
 h=r-k+2.
\]
Indeed, if at least $h$ vertices of every base edge remain private, then at most $k-2$ base vertices are shared.  A putative $k$-crown must therefore attach at least two petals at private vertices, forcing those two petals into one copy of an intersecting block, where they cannot be disjoint.

This gives the block parameter
\begin{equation}\label{eq:rho}
 \rho_{r,k}:=
 \sup_G \frac{|E(G)|}{\tau_{r-k+2}(G)},
\end{equation}
where the supremum is over finite linear intersecting $r$-uniform hypergraphs without isolated vertices.  The parameter $\rho_{r,k}$ controls a family of explicit lower-bound constructions and admits an optimal general upper bound.

A port-shadow formulation of the amalgamation gives an estimate for every sufficiently large number of vertices.  For every fixed block $G$ and $h$-fold transversal $B$,
\begin{equation}\label{eq:intro-allorders}
 \exlin_r(n,C^r_{1,k})
 \ge \frac{|E(G)|}{|B|}n-O_{G,B}(\sqrt n).
\end{equation}
Incidence counting gives the bound
\begin{equation}\label{eq:intro-ceiling}
 \rho_{r,k}\le \frac{r}{r-k+2}.
\end{equation}
The equality case is rigid: after dualization it becomes a pairwise balanced design with a distinguished regular subfamily.  It also gives a size bound for equality blocks.

Finite geometry nearly saturates \eqref{eq:intro-ceiling} in every projective order.  Let $r=q+1$ and put
\[
 h=q-k+3.
\]
Deleting one point of $\PG(2,q)$ together with all lines through it gives a truncated projective plane $T_q$ with
\[
 \tau_h(T_q)=hq,
\]
and therefore
\begin{equation}\label{eq:intro-sandwich}
 \frac qh\le \rho_{q+1,k}\le \frac{q+1}{h}.
\end{equation}
The multiplicative gap is exactly $(q+1)/q$, independent of $k$.

Maximal arcs identify exact equality cases.  A maximal arc of degree $h$ in $\PG(2,q)$ is a point set meeting every projective line in either $0$ or $h$ points.  The hypergraph formed by its secant lines attains the upper value $(q+1)/h$.  Denniston's construction supplies maximal arcs of every degree $h$ dividing $q$ when $q$ is even \cite{Denniston1969,DeClerckDeWinterMaes2011}; nontrivial maximal arcs do not exist in Desarguesian planes of odd order \cite{BallBlokhuisMazzocca1997}.  Hence, for even prime-power $q$ and $h\mid q$,
\begin{equation}\label{eq:intro-exact}
 \rho_{q+1,k}=\frac{q+1}{h}.
\end{equation}
In particular, for the full crown,
\[
 \frac q2\le \rho_{q+1,q+1}\le \frac{q+1}{2}
\]
for every prime-power $q$, and equality with the upper value holds for every even $q$.  Thus the best coefficient furnished by intersecting-block amalgamations is asymptotically $q/2$.

Padding truncated projective planes with private degree-one vertices removes the restriction $r=q+1$.  This yields
\[
 \rho_{r,r}=(1-o(1))\frac r2
\]
as $r\to\infty$, and uniformly for every fixed $\varepsilon>0$ and $\varepsilon r\le k\le r$,
\[
 \rho_{r,k}=(1+o(1))\frac{r}{r-k+2}.
\]
These statements optimize the intersecting-block method asymptotically in every linear-size crown regime.  The unrestricted coefficient $\clin_r(C^r_{1,k})$ is a separate extremal parameter and may be larger.

Adak and Verma gave an affine-plane construction for the three-petal crown \cite{AdakVerma2026}; this is the case $h=q$ of the optimization problem \eqref{eq:rho}.

\section{Finite geometry and design preliminaries}\label{sec:prelim}

A hypergraph is \emph{intersecting} if every two distinct edges meet.  It is a \emph{star} if all edges contain one common vertex.  For a vertex $v$ of $G$, write $d_G(v)$ for its degree and $\Delta(G)$ for the maximum degree.

A projective plane of order $q$ has $q^2+q+1$ points and the same number of lines.  Each line contains $q+1$ points, each point lies on $q+1$ lines, and every two distinct points and every two distinct lines determine a unique line and point, respectively.  For every prime power $q$, the Desarguesian plane $\PG(2,q)$ exists.

An affine plane of order $K$ has $K^2$ points.  Its lines split into $K+1$ parallel classes.  Each parallel class consists of $K$ pairwise disjoint lines that partition the point set, and two distinct affine points lie on a unique line.  For every prime power $K$, the Desarguesian affine plane $\AG(2,K)$ exists.  Standard background on finite planes can be found in \cite{Dembowski1968}.

A \emph{pairwise balanced design} on a point set $X$ is a family of blocks such that every two distinct points of $X$ lie in exactly one common block.  The \emph{replication number} of a point is the number of blocks containing it.  We allow repeated singleton blocks; they do not affect the pair condition and will encode degree-one vertices after dualization.

\begin{definition}[Maximal arc]\label{def:maximal-arc}
Let $\Pi$ be a projective plane of order $q$.  A nonempty point set $A\subseteq P(\Pi)$ is a \emph{maximal arc of degree $h$} if every line of $\Pi$ meets $A$ in either $0$ or $h$ points.
\end{definition}

For $2\le h\le q$, counting incident pairs $(a,L)$ with $a\in A\cap L$ and pairs of points of $A$ on secants gives the standard size formula
\begin{equation}\label{eq:maxarc-size}
 |A|=q(h-1)+h.
\end{equation}
If $s$ is the number of secants, then
\[
 sh=(q+1)|A|,
 \qquad
 s\binom h2=\binom{|A|}{2}.
\]
Eliminating $s$ gives \eqref{eq:maxarc-size}.

\section{Constrained ports and the all-orders transfer}\label{sec:ports}

Let $G=(P,\calE)$ be a finite linear $r$-uniform hypergraph and let $B\subseteq P$.  Put $U=P\setminus B$.  The vertices in $B$ will remain private to each copy of $G$, while the vertices in $U$ will be shared through ports.  The only pairs of ports that must be separated are pairs that can occur together in one block edge.

\begin{definition}[Port-shadow graph]\label{def:shadow}
The \emph{port-shadow graph} $J_G(U)$ has vertex set $U$, with $uv\in E(J_G(U))$ if and only if some edge of $G$ contains both $u$ and $v$.
\end{definition}

\begin{definition}[Constrained port array]\label{def:constrained-port}
Let $J$ be a graph on $U$ and let $X$ be a finite copy set.  A \emph{$J$-constrained port array} on $X$ is a family of maps
\[
 \psi_u:X\longrightarrow Y_u\qquad(u\in U)
\]
such that for every edge $uv\in E(J)$, the map
\[
 x\longmapsto (\psi_u(x),\psi_v(x))
\]
is injective.
\end{definition}

Equivalently, for two distinct rows $x,y\in X$, the set of coordinates $u$ on which $\psi_u(x)=\psi_u(y)$ must be an independent set of $J$.  When $J$ is complete, this is a mixed-alphabet packing array of strength two; see \cite{StevensMendelsohn2004}.

\begin{theorem}[Constrained-port amalgamation]\label{thm:constrained}
Let $3\le k\le r$, put $h=r-k+2$, and let $G=(P,\calE)$ be a finite linear intersecting $r$-uniform hypergraph.  Let $B\subseteq P$ be an $h$-fold transversal, put $U=P\setminus B$, and let $(\psi_u)_{u\in U}$ be a $J_G(U)$-constrained port array on a finite set $X$.  Then there is a linear $r$-uniform $C^r_{1,k}$-free hypergraph $H$ with
\begin{equation}\label{eq:constrained-counts}
 |E(H)|=|X|\,|E(G)|,
 \qquad
 |V(H)|\le |X|\,|B|+\sum_{u\in U}|Y_u|.
\end{equation}
\end{theorem}

\begin{proof}
For each $x\in X$ and each $p\in B$, create a private vertex $(x,p)$.  For $u\in U$, represent the copy of $u$ in row $x$ by the port vertex $(u,\psi_u(x))$.  Thus define
\[
 \phi_x(p)=
 \begin{cases}
  (x,p),&p\in B,\\
  (p,\psi_p(x)),&p\in U.
 \end{cases}
\]
The tags $p$ make $\phi_x$ injective.  For $E\in\calE$, put $E_x=\phi_x(E)$, and let
\[
 E(H)=\{E_x:E\in\calE,\ x\in X\}.
\]
Every edge has size $r$.

We prove linearity.  Two edges in the same copy meet in at most one vertex because $G$ is linear.  Suppose $x\ne y$ and that $E_x$ and $F_y$ have two common vertices.  Common vertices cannot be private, so their labels are two distinct ports $u,v\in U$.  If $E\ne F$, then $u,v\in E\cap F$, contradicting linearity of $G$.  If $E=F$, then $u$ and $v$ occur together in an edge of $G$, so $uv\in E(J_G(U))$.  The equalities $\psi_u(x)=\psi_u(y)$ and $\psi_v(x)=\psi_v(y)$ contradict the injectivity condition in Definition~\ref{def:constrained-port}.  Hence $H$ is linear.

Now fix a possible crown base $E_x$.  Because $B$ is an $h$-fold transversal,
\[
 |E\setminus B|\le r-h=k-2.
\]
A $k$-crown uses $k$ distinct attachment vertices on its base, so at least two attachments would lie at private vertices $(x,p)$ with $p\in B\cap E$.  Any edge containing such a private vertex belongs to copy $x$.  Thus the two corresponding petals are images of two distinct edges of the intersecting hypergraph $G$ inside the same copy; they meet, contradicting the definition of a crown.  Therefore $H$ is crown-free.

Finally, there are $|X||B|$ private vertices and at most $\sum_{u\in U}|Y_u|$ port vertices.  Each pair $(x,E)$ contributes one edge, and distinct pairs cannot produce the same $r$-set in a linear hypergraph.  This gives \eqref{eq:constrained-counts}.
\end{proof}

An affine-plane instance is obtained by properly coloring $J_G(U)$ with at most $K+1$ colors and assigning one parallel class of $\AG(2,K)$ to each color.  For every port $u$, let $\psi_u(x)$ be the affine line in the assigned class through $x$, with the symbol tagged by $u$.  If $uv$ is a shadow edge, then $u$ and $v$ receive different directions; two distinct affine directions through a point determine that point, so the required pair map is injective.  This gives $K^2$ rows and $|Y_u|=K$ for every port.  It is therefore enough that
\[
 K+1\ge \chi(J_G(U)).
\]
If the shadow is empty, every $\psi_u$ may instead be constant.

\begin{corollary}[All-orders transfer]\label{cor:allorders}
Let $G$ and $B$ satisfy the hypotheses of Theorem~\ref{thm:constrained}, and write $m=|E(G)|$ and $b=|B|$.  Then for all sufficiently large $n$,
\begin{equation}\label{eq:allorders}
 \exlin_r(n,C^r_{1,k})
 \ge \frac{m}{b}n-O_{G,B}(\sqrt n).
\end{equation}
Consequently,
\begin{equation}\label{eq:liminf-rho}
 \liminf_{n\to\infty}
 \frac{\exlin_r(n,C^r_{1,k})}{n}
 \ge \rho_{r,k}.
\end{equation}
If $J_G(P\setminus B)$ is empty, the error term in \eqref{eq:allorders} improves to $O_G(1)$.
\end{corollary}

\begin{proof}
Let $U=P\setminus B$ and put $s=|U|$.  Set
\[
 M=\left\lceil\sqrt{n/b}\right\rceil.
\]
For all sufficiently large $n$, Bertrand's postulate gives a prime $K$ with
\[
 M\le K<2M,
\]
and $K+1\ge\chi(J_G(U))$.  Use the affine-plane port array described above, but retain only
\[
 N=\left\lfloor\frac{n-sK}{b}\right\rfloor
\]
rows.  Since $K^2\ge M^2\ge n/b\ge N$, enough rows are available.  The resulting hypergraph has at most
\[
 bN+sK\le n
\]
vertices and exactly $mN$ edges.  Adding isolated vertices if necessary makes the vertex count exactly $n$.  Because $K=O(\sqrt n)$,
\[
 N=\frac nb-O_{G,B}(\sqrt n),
\]
which proves \eqref{eq:allorders}.

If the shadow is empty, take arbitrary $N=\lfloor(n-s)/b\rfloor$ rows and constant port maps.  Then the construction uses at most $bN+s$ vertices and has $mN=(m/b)n-O_G(1)$ edges.

For \eqref{eq:liminf-rho}, first take $B$ to be a minimum $h$-fold transversal of a fixed $G$, then take the liminf in \eqref{eq:allorders}, and finally take the supremum over $G$ in \eqref{eq:rho}.
\end{proof}

\section{Efficiency bound and equality}\label{sec:efficiency}

The ratio defining $\rho_{r,k}$ satisfies a general upper bound.  The proof also gives an exact defect identity that records every source of inefficiency.

\begin{theorem}[Efficiency bound]\label{thm:ceiling}
Let $3\le k\le r$ and put $h=r-k+2$.  Then
\begin{equation}\label{eq:ceiling}
 \rho_{r,k}\le \frac rh.
\end{equation}
More precisely, let $G$ be a finite linear intersecting $r$-uniform hypergraph with $m$ edges and let $B$ be an $h$-fold transversal.  If $G$ is not a star, then
\begin{equation}\label{eq:defect}
 r|B|-hm
 =\sum_{v\in B}\bigl(r-d_G(v)\bigr)
 +\sum_{E\in E(G)}\bigl(|E\cap B|-h\bigr).
\end{equation}
In the non-star case, equality $m/|B|=r/h$ holds if and only if every $v\in B$ has degree $r$ and every edge meets $B$ in exactly $h$ vertices.
\end{theorem}

\begin{proof}
If $G$ is not a star, then
\begin{equation}\label{eq:maxdeg}
 \Delta(G)\le r.
\end{equation}
Indeed, suppose a vertex $v$ lies in more than $r$ edges.  If some edge $F$ omitted $v$, then, because $G$ is intersecting, $F$ would have to meet every edge through $v$.  These intersection points would be distinct: two edges through $v$ cannot share a second vertex in a linear hypergraph.  Thus the $r$ vertices of $F$ would have to contain more than $r$ distinct intersection points, a contradiction.  Hence every edge contains $v$, so $G$ is a star.

Assume first that $G$ is not a star.  By \eqref{eq:maxdeg}, incidence counting gives
\[
 hm
 \le \sum_{E\in E(G)}|E\cap B|
 =\sum_{v\in B}d_G(v)
 \le r|B|.
\]
Thus $m/|B|\le r/h$.  Moreover, subtracting the common incidence sum from both ends gives exactly \eqref{eq:defect}.  Every term on the right is nonnegative, so equality holds precisely under the two stated conditions.

Now suppose $G$ is a star with center $v$ and $m$ edges.  Because $G$ is linear, the sets $E\setminus\{v\}$ are pairwise disjoint.  Since $h\ge2$, a minimum $h$-fold transversal consists of the center together with $h-1$ additional vertices from each edge, and therefore has size
\[
 1+(h-1)m.
\]
Consequently
\[
 \frac{m}{\tau_h(G)}
 =\frac{m}{1+(h-1)m}
 <\frac{1}{h-1}
 \le \frac rh.
\]
This proves \eqref{eq:ceiling} in all cases.
\end{proof}

The equality conditions have a dual interpretation.  This is not needed for the numerical ceiling, but it explains why design theory and projective geometry appear in the extremizers.

\begin{theorem}[Equality structure]\label{thm:dual-normal}
Suppose a non-star pair $(G,B)$ attains equality in Theorem~\ref{thm:ceiling}.  Put $m=|E(G)|$ and, on point set $E(G)$, define
\[
 Q_v:=\{E\in E(G):v\in E\}
 \qquad(v\in V(G)).
\]
Then the family $\{Q_v:v\in V(G)\}$ is a pairwise balanced design with the following properties:
\begin{enumerate}[label=(\roman*)]
 \item every design point has replication number $r$;
 \item every design block has size at most $r$;
 \item the blocks $Q_v$ with $v\in B$ all have size $r$, and every design point lies in exactly $h$ of these selected blocks.
\end{enumerate}
Conversely, dualizing any pairwise balanced design with properties (i)--(iii) gives an equality pair.  In addition,
\begin{equation}\label{eq:m-bound}
 m\le r(r-1).
\end{equation}
If $m=r(r-1)$, then $h=r-1$; adjoining one edge to $G$ produces a projective plane of order $r-1$, and $B$ is the complement of the adjoined edge.
\end{theorem}

\begin{proof}
Because $G$ is linear and intersecting, any two distinct edges of $G$ meet in exactly one vertex.  Hence any two distinct points of the dual point set $E(G)$ lie in exactly one common block $Q_v$.  A dual point corresponding to an edge $E$ lies in the $r$ blocks indexed by the vertices of $E$, proving (i).  The non-star degree bound \eqref{eq:maxdeg} gives (ii).  Equality in Theorem~\ref{thm:ceiling} says that every $v\in B$ has degree $r$ and every original edge contains exactly $h$ vertices of $B$, which is exactly (iii).  The converse follows by reversing this dualization.

Fix a dual point $x$.  Its other $m-1$ points are partitioned among the $r$ blocks containing $x$.  Since each such block has size at most $r$,
\[
 m-1\le r(r-1).
\]
If $m=r(r-1)+1$, then equality holds in this partition bound for every $x$, so every dual block has size $r$.  The dual is then a projective plane of order $r-1$.  Property (iii) gives the selected-block incidence equation
\[
 r|B|=hm=h(r^2-r+1).
\]
But $1\le h\le r-1$ and
\[
 \gcd(r,r^2-r+1)=1,
\]
so the right side is not divisible by $r$, a contradiction.  Therefore \eqref{eq:m-bound} holds.

Assume now that $m=r(r-1)$.  For any dual point $x$, the total capacity of its $r$ incident blocks is $r(r-1)$ other-point incidences, whereas only $m-1=r(r-1)-1$ other points must be covered.  Since the incident blocks partition the other points, there is exactly one unit of deficit: $x$ lies in exactly one block of size $r-1$ and in $r-1$ blocks of size $r$.  Thus the size-$(r-1)$ blocks partition the $m$ dual points.  There are
\[
 \frac{m}{r-1}=r
\]
such blocks.

Return to the original hypergraph.  The $r$ vertices corresponding to these size-$(r-1)$ dual blocks each have degree $r-1$, and every original edge contains exactly one of them.  Adjoin a new edge consisting of these $r$ vertices.  The new edge meets every old edge exactly once.  After adjoining it, every vertex has degree $r$, every edge has size $r$, and every two edges meet exactly once.  The resulting incidence structure has $r^2-r+1$ lines and, by counting incidences, the same number of points.  Linearity implies that no pair of points lies on two lines, while
\[
 (r^2-r+1)\binom r2=\binom{r^2-r+1}{2}.
\]
Thus every pair of points lies on a line, and the completed structure is a projective plane of order $r-1$.

The new edge is disjoint from $B$, because vertices of $B$ had degree $r$ already in $G$.  Every old line meets $B$ in exactly $h$ points.  Hence $B$ is a maximal arc of degree $h$ in the completed plane.  By \eqref{eq:maxarc-size},
\[
 |B|=(r-1)(h-1)+h.
\]
On the other hand, equality $r|B|=hm$ gives $|B|=h(r-1)$.  Equating the two expressions yields $h=r-1$.  In that case $|B|=(r-1)^2$, so $B$ is exactly the complement of the new line.
\end{proof}

\section{Maximal arcs and equality}\label{sec:maxarcs}

The universal ceiling is attained by selecting the correct line subfamily of a projective plane.  The relevant line subfamilies are precisely the secants of maximal arcs.

\begin{proposition}[Equality within a projective plane]\label{prop:projective-equality}
Let $\Pi$ be a projective plane of order $q$, let $\calL'\subseteq\calL(\Pi)$, and let $G=(P(\Pi),\calL')$.  If $B$ is an $h$-fold transversal of $G$, where $2\le h\le q$, then
\begin{equation}\label{eq:projective-bound}
 \frac{|\calL'|}{|B|}\le \frac{q+1}{h}.
\end{equation}
Equality holds if and only if $B$ is a maximal arc of degree $h$ in $\Pi$ and $\calL'$ is exactly the set of lines meeting $B$ in $h$ points.
\end{proposition}

\begin{proof}
Count incidences between $B$ and the selected lines.  Because $B$ is an $h$-fold transversal,
\[
 h|\calL'|\le I(B,\calL').
\]
Each point of $B$ lies on at most $q+1$ selected lines, so
\[
 I(B,\calL')\le(q+1)|B|.
\]
This proves \eqref{eq:projective-bound}.

Suppose equality holds.  Then every selected line contains exactly $h$ points of $B$, and every point of $B$ lies on all $q+1$ projective lines through it, so every projective line meeting $B$ is selected.  Therefore an unselected line is disjoint from $B$, while a selected line meets $B$ in exactly $h$ points.  Thus $B$ is a maximal arc and $\calL'$ is its secant set.  The converse is immediate from the same incidence count.
\end{proof}

Let $A$ be a maximal $h$-arc and let $G_A$ be the hypergraph whose edges are the secant lines of $A$.  From incidence counting,
\begin{equation}\label{eq:secant-count}
 |E(G_A)|=\frac{q+1}{h}|A|.
\end{equation}
Since $A$ is an $h$-fold transversal, $\tau_h(G_A)\le|A|$.  Theorem~\ref{thm:ceiling} gives the reverse conclusion: if $\tau_h(G_A)<|A|$, then \eqref{eq:secant-count} would make $|E(G_A)|/\tau_h(G_A)>(q+1)/h$, contradicting the ceiling for $r=q+1$.  Hence
\begin{equation}\label{eq:arc-tau}
 \tau_h(G_A)=|A|.
\end{equation}

\begin{theorem}[Maximal-arc construction]\label{thm:maxarc-rho}
Let $q$ be a prime power, let $3\le k\le q+1$, and put $h=q-k+3$.  If $\PG(2,q)$ contains a maximal arc of degree $h$, then
\begin{equation}\label{eq:maxarc-rho}
 \rho_{q+1,k}=\frac{q+1}{h}.
\end{equation}
In particular,
\begin{align}
 \rho_{q+1,3}&=\frac{q+1}{q}
 &&\text{for every prime power }q,\label{eq:k3-exact}\\
 \rho_{q+1,k}&=\frac{q+1}{q-k+3}
 &&\text{if $q$ is even and $q-k+3\mid q$},\label{eq:even-general-exact}\\
 \rho_{q+1,q+1}&=\frac{q+1}{2}
 &&\text{for every even prime power }q.\label{eq:full-even}
\end{align}
\end{theorem}

\begin{proof}
The secant-line hypergraph of the maximal arc has ratio $(q+1)/h$ by \eqref{eq:secant-count}--\eqref{eq:arc-tau}, while Theorem~\ref{thm:ceiling} gives the reverse inequality.

For $k=3$, one has $h=q$.  The complement of any projective line is a maximal arc of degree $q$, so \eqref{eq:k3-exact} holds for every $q$.  If $q$ is even, Denniston constructed maximal arcs of every degree $h$ dividing $q$ in $\PG(2,q)$ \cite{Denniston1969}; this gives \eqref{eq:even-general-exact}, and $h=2$ gives \eqref{eq:full-even}.
\end{proof}

The case $h=q$ has an especially simple port structure.  Let $L_\infty$ be a projective line, delete it from the edge set, and let $B=P\setminus L_\infty$.  Every retained line has $q$ private vertices in $B$ and one port on $L_\infty$.  No retained edge contains two ports, so the port shadow is empty.  Taking $N$ arbitrary copies with constant port maps gives
\begin{equation}\label{eq:k3-count}
 |E(H)|=Nq(q+1),
 \qquad
 |V(H)|=Nq^2+q+1.
\end{equation}
Thus the coefficient tends to $(q+1)/q$ with only constant port overhead.  Related affine-type lower constructions appear in the recent crown and hypertree literature \cite{ZhangBroersmaWang2025,AdakVerma2026}.  Equation~\eqref{eq:k3-exact} identifies this construction as an exact optimizer over the entire intersecting-block class.

\section{Truncated projective planes}\label{sec:dual-affine}

Maximal arcs give equality only in special parameter ranges.  A different line deletion gives a block that works for every $h$ and misses the universal ceiling by only $1/h$ additively.

Fix a projective plane $\Pi$ of order $q$ and a point $z$.  Delete $z$ and delete from the edge set all $q+1$ lines through $z$.  The remaining hypergraph will be denoted by $T_q$.  It has
\[
 q(q+1)\text{ vertices},\qquad q^2\text{ edges},
\]
and is linear, intersecting, and $(q+1)$-uniform.  The $q+1$ deleted lines through $z$ determine classes
\[
 C_L:=L\setminus\{z\}\qquad(L\ni z),
\]
which partition $V(T_q)$ into $q+1$ parts of size $q$.  Every edge of $T_q$ contains exactly one vertex from each part.  This incidence structure is dual to an affine plane.

\begin{theorem}[Multiple transversals of $T_q$]\label{thm:Tq-tau}
For every $1\le h\le q+1$,
\begin{equation}\label{eq:Tq-tau}
 \tau_h(T_q)=hq.
\end{equation}
\end{theorem}

\begin{proof}
The union of any $h$ classes $C_L$ meets every edge in exactly $h$ points and has size $hq$, so $\tau_h(T_q)\le hq$.

Conversely, every vertex of $T_q$ lies on exactly $q$ retained projective lines: among the $q+1$ projective lines through the vertex, exactly the line joining it to $z$ was deleted.  If $B$ is an $h$-fold transversal, then
\[
 hq^2
 \le \sum_{E\in E(T_q)}|E\cap B|
 =\sum_{v\in B}d_{T_q}(v)
 =q|B|.
\]
Thus $|B|\ge hq$, proving \eqref{eq:Tq-tau}.
\end{proof}

\begin{theorem}[Truncated-plane bounds]\label{thm:sandwich}
Let $q$ be a prime power, let $3\le k\le q+1$, and put $h=q-k+3$.  Then
\begin{equation}\label{eq:sandwich}
 \frac qh\le \rho_{q+1,k}\le \frac{q+1}{h}.
\end{equation}
Equivalently,
\begin{equation}\label{eq:uniform-asymptotic}
 \rho_{q+1,k}
 =\left(1+O(q^{-1})\right)\frac{q+1}{q-k+3},
\end{equation}
with relative error uniform over $3\le k\le q+1$.
\end{theorem}

\begin{proof}
By Theorem~\ref{thm:Tq-tau},
\[
 \frac{|E(T_q)|}{\tau_h(T_q)}
 =\frac{q^2}{hq}
 =\frac qh,
\]
which gives the lower bound.  The upper bound is Theorem~\ref{thm:ceiling}.  The ratio of the upper bound to the lower bound is $(q+1)/q$.
\end{proof}

The port-shadow improvement is particularly transparent for $T_q$.  Choose a minimum $h$-fold transversal equal to the union of $h$ partite classes.  The ports are the vertices in the other $q+1-h$ classes.  Two ports in the same class do not occur together in a retained edge, while any two ports in different classes do.  Thus the port shadow is the complete $(q+1-h)$-partite graph with parts of size $q$, and
\[
 \chi(J)=q+1-h.
\]
It is therefore enough to take the auxiliary affine-plane order with $K+1\ge q+1-h$.  Assigning a separate direction to every port would require $K+1\ge q(q+1-h)$.

For the full crown, $h=2$, so Corollary~\ref{cor:allorders} and Theorem~\ref{thm:sandwich} give
\begin{equation}\label{eq:full-liminf}
 \liminf_{n\to\infty}
 \frac{\exlin_{q+1}(n,C^{q+1}_{1,q+1})}{n}
 \ge \frac q2
\end{equation}
for every prime-power $q$.  If $q$ is even, Theorem~\ref{thm:maxarc-rho} improves the right side to $(q+1)/2$ and proves that no intersecting block can do better.  For odd $q$, the block optimum lies in an interval of width $1/2$:
\[
 \frac q2\le \rho_{q+1,q+1}\le \frac{q+1}{2}.
\]
By comparison, the unrestricted upper bound \eqref{eq:adak-upper-intro} gives
\begin{equation}\label{eq:full-unrestricted-upper}
 \clin_{q+1}(C^{q+1}_{1,q+1})
 \le \frac{q^2+1}{q+1}
 =q-1+\frac{2}{q+1}.
\end{equation}
The unrestricted Tur\'an problem therefore retains a factor-of-two-scale gap at the full-crown endpoint.

\section{Padding to arbitrary uniformities}\label{sec:padding}

The preceding construction uses uniformity $q+1$.  It can be lifted to a larger uniformity by adding private degree-one vertices to each edge.

Let $q\le r-1$ be a prime power.  Form $T_q^{(r)}$ from $T_q$ by adjoining $r-q-1$ new vertices to every edge, with the new vertices used for different edges pairwise disjoint.  We call the original vertices \emph{old} and the new degree-one vertices \emph{leaves}.

\begin{proposition}[Transversal number under padding]\label{prop:padded}
For $1\le h\le r$,
\begin{equation}\label{eq:padded-tau}
 \tau_h(T_q^{(r)})=
 \begin{cases}
  hq,&h\le q+1,\\[0.25em]
  q^2(h-q)+q,&h\ge q+1.
 \end{cases}
\end{equation}
In particular, if $h=r-k+2\le q+1$, then
\begin{equation}\label{eq:padded-rho}
 \rho_{r,k}\ge \frac qh.
\end{equation}
\end{proposition}

\begin{proof}
Suppose first that $h\le q+1$.  The union of $h$ old partite classes has size $hq$ and meets every padded edge in $h$ old vertices, so $\tau_h\le hq$.  Every old vertex has degree $q$ and every leaf has degree $1\le q$.  Hence any $h$-fold transversal $B$ satisfies
\[
 hq^2
 \le \sum_{v\in B}d(v)
 \le q|B|,
\]
which gives $|B|\ge hq$.

Now suppose $h\ge q+1$.  Let $B$ be an $h$-fold transversal and let $S$ be the set of selected old vertices.  For each edge $E$, put $a_E=|S\cap E|$.  Since the leaves of distinct edges are private, at least $h-a_E$ leaves of $E$ must be selected.  Therefore
\[
 |B|
 \ge |S|+\sum_E(h-a_E).
\]
There are $q^2$ edges and every old vertex has degree $q$, so $\sum_E a_E=q|S|$.  Thus
\[
 |B|\ge hq^2-(q-1)|S|.
\]
The number of old vertices is $q(q+1)$, and the coefficient of $|S|$ is negative.  Hence
\[
 |B|
 \ge hq^2-(q-1)q(q+1)
 =q^2(h-q)+q.
\]
Equality is attained by selecting all old vertices and exactly $h-q-1$ leaves from every edge.  This proves \eqref{eq:padded-tau}.  Since $T_q^{(r)}$ still has $q^2$ edges, the first case gives $|E|/\tau_h=q/h$, proving \eqref{eq:padded-rho}.
\end{proof}

Let $p(r)$ denote the largest prime at most $r-1$.  The prime number theorem implies
\begin{equation}\label{eq:prime-gap}
 p(r)=(1-o(1))r.
\end{equation}
Taking $q=p(r)$ in Proposition~\ref{prop:padded} yields the asymptotic optimization below.

\begin{corollary}[Asymptotics for arbitrary uniformity]\label{cor:alluniform}
As $r\to\infty$,
\begin{equation}\label{eq:all-r-full}
 \rho_{r,r}=\left(1-o(1)\right)\frac r2.
\end{equation}
More generally, for every fixed $\varepsilon>0$, uniformly over $\varepsilon r\le k\le r$,
\begin{equation}\label{eq:all-r-k}
 \rho_{r,k}
 =\left(1+o(1)\right)\frac{r}{r-k+2}.
\end{equation}
\end{corollary}

\begin{proof}
For the full crown, $h=2$.  Proposition~\ref{prop:padded} with $q=p(r)$ gives
\[
 \rho_{r,r}\ge\frac{p(r)}2,
\]
while Theorem~\ref{thm:ceiling} gives $\rho_{r,r}\le r/2$.  Equation \eqref{eq:prime-gap} proves \eqref{eq:all-r-full}.

Now fix $\varepsilon>0$ and assume $\varepsilon r\le k\le r$.  Then
\[
 h=r-k+2\le(1-\varepsilon)r+2.
\]
By \eqref{eq:prime-gap}, for all sufficiently large $r$ this is at most $p(r)+1$.  Proposition~\ref{prop:padded} therefore gives
\[
 \frac{p(r)}h\le\rho_{r,k}\le\frac rh.
\]
The quotient of the two endpoints is $r/p(r)=1+o(1)$, independent of $k$ in the stated range, proving \eqref{eq:all-r-k} uniformly.
\end{proof}

\section{Finite examples and a rank obstruction}\label{sec:finite-rank}

The asymptotic theory has two concrete consequences worth separating from the optimization results.  The first is a small connected counterexample to the coefficient $r/(r-1)$ in uniformity four.  The second shows that a natural incidence-rank inequality cannot hold globally for full-crown-free hypergraphs.

\subsection{A connected 4-uniform example}

For $q=3$, the following triangle-block construction gives a smaller connected example than the direct truncated-plane affine-array construction.

\begin{proposition}\label{prop:469}
There exists a connected linear $4$-uniform $C^4_{1,4}$-free hypergraph with $469$ vertices and $637$ edges.  In particular,
\[
 \frac{637}{469}>\frac43.
\]
\end{proposition}

\begin{proof}
In $\PG(2,3)$, take three nonconcurrent lines and let $B$ be their union.  The three lines have $3(3+1)-3=9$ points in total.  Every projective line meets $B$ at least twice: a line through a triangle vertex meets the third side at another point, while a line through no triangle vertex meets all three sides in distinct points.  Hence $B$ is a double transversal.  The projective plane has $13$ points, so there are $4$ ports.

Use the basic affine-direction instance of Theorem~\ref{thm:constrained} with the four ports assigned distinct parallel classes of $\AG(2,7)$.  There are $7^2=49$ copies of the $13$-edge projective-plane block.  The private vertices contribute $9\cdot49$, and the four port coordinates contribute $4\cdot7$.  Thus
\[
 |V|=9\cdot49+4\cdot7=469,
 \qquad
 |E|=13\cdot49=637.
\]
Theorem~\ref{thm:constrained} gives linearity and crown-freeness.  Connectivity follows because at least two distinct affine directions are used: any two copies can be joined through an intermediate copy lying at the intersection of one line in each of two selected directions.  Finally,
\[
 3\cdot637=1911>1876=4\cdot469.
\]
\end{proof}

\subsection{A global incidence-rank obstruction}

Let $N(H)$ be the edge-by-vertex incidence matrix of a finite hypergraph $H$, over $\R$.  The analogous global incidence-rank inequality with coefficient $(q+1)/q$ is false for full crowns.

\begin{corollary}[Rank obstruction]\label{cor:rank}
For every prime power $q\ge3$, there is a finite linear $C^{q+1}_{1,q+1}$-free hypergraph $H$ such that
\begin{equation}\label{eq:rank-fail}
 (q+1)\rank_{\R}N(H)<q|E(H)|.
\end{equation}
\end{corollary}

\begin{proof}
By \eqref{eq:full-liminf}, there are finite full-crown-free hypergraphs with edge--vertex ratio arbitrarily close to at least $q/2$.  For $q\ge3$,
\[
 \frac q2>\frac{q+1}{q},
\]
because $q^2>2q+2$.  Choose such an $H$ with
\[
 |E(H)|>\frac{q+1}{q}|V(H)|.
\]
Since $\rank N(H)\le|V(H)|$,
\[
 (q+1)\rank N(H)
 \le(q+1)|V(H)|
 <q|E(H)|.
\]
\end{proof}

This obstruction is caused by global port sharing.  Projective geometry may describe a local equality configuration, while identifications among many blocks increase the global crown-free density.  Any endpoint incidence-rank argument must account for this amalgamation mechanism.

\section{Nonintersecting templates and safe blockers}\label{sec:safe}

The $h$-fold transversal condition uses the hypothesis that the block is intersecting only once: two petals attached at private vertices are forced into the same copy, and intersection prevents them from being disjoint.  For a nonintersecting block, the exact local quantity is the maximum number of pairwise disjoint private petals that one base edge can support.

Let $G=(P,\calE)$ be a finite linear $r$-uniform hypergraph, not necessarily intersecting, and let $B\subseteq P$.  For $E\in\calE$, define $\mu_B(E)$ to be the largest integer $t$ for which there are pairwise disjoint edges
\[
 F_1,\ldots,F_t\in\calE\setminus\{E\}
\]
and distinct points $b_1,\ldots,b_t\in B\cap E$ such that
\[
 F_i\cap E=\{b_i\}\qquad(1\le i\le t).
\]

\begin{definition}[Safe blocker]\label{def:safe}
A set $B\subseteq P$ is \emph{$k$-safe} if, for every $E\in\calE$,
\begin{equation}\label{eq:safe}
 |E\setminus B|+\mu_B(E)\le k-1.
\end{equation}
Equivalently, if $h=r-k+2$, then
\[
 |E\cap B|-\mu_B(E)\ge h-1
\]
for every edge $E$.
\end{definition}

\begin{theorem}[Safe-block transfer and affine converse]\label{thm:safe}
The conclusion of Theorem~\ref{thm:constrained} remains valid if the hypotheses ``$G$ is intersecting and $B$ is an $h$-fold transversal'' are replaced by ``$G$ is linear and $B$ is $k$-safe.''

Conversely, consider the canonical affine construction in which distinct ports are assigned distinct parallel classes of $\AG(2,K)$ and every affine point is used as a copy.  If $B$ is not $k$-safe and
\begin{equation}\label{eq:safe-K}
 K-1>(k-1)r,
\end{equation}
then the resulting amalgamation contains a copy of $C^r_{1,k}$.
\end{theorem}

\begin{proof}
The linearity proof is unchanged.  Fix a possible base $E_x$.  A petal attached at a private vertex of $E_x$ lies in copy $x$.  The preimages of all such private petals are pairwise disjoint edges of $G$ meeting $E$ at distinct points of $B\cap E$, so there are at most $\mu_B(E)$ of them.  At most one additional petal can attach at each port vertex of the base, giving at most $|E\setminus B|$ port petals.  Condition \eqref{eq:safe} therefore forbids $k$ petals.

For the converse, suppose $B$ is unsafe for an edge $E$, and fix a copy indexed by an affine point $x$.  Put
\[
 s=\min\{\mu_B(E),k\}.
\]
Choose $s$ pairwise disjoint template edges from a witnessing family in the definition of $\mu_B(E)$; their images in copy $x$ give $s$ pairwise disjoint private petals.  If $s=k$, these petals already form the required crown with base $E_x$.

Assume therefore that $s=\mu_B(E)<k$.  Since $B$ is unsafe,
$|E\setminus B|\ge k-s$.  Choose distinct ports
$u_1,\ldots,u_{k-s}\in E\setminus B$.  For each $u_j$, choose a point
$y_j\ne x$ on the affine line through $x$ in the direction assigned to
$u_j$, and use the copy of $E$ indexed by $y_j$ as an external petal.
It meets the base $E_x$ exactly at the $u_j$-port because distinct ports
use distinct directions.  Choose the points $y_j$ greedily.  Each previously
chosen external petal has $r$ vertices, and for each such vertex at most one
of the $K-1$ candidates on the $u_j$-line can make the new copy meet that
vertex.  Before the $k$th petal is chosen, fewer than $(k-1)r$ candidates
are therefore forbidden, so \eqref{eq:safe-K} guarantees a valid choice.

Every external petal is disjoint from every private petal: a private
template petal meets $E$ only at its private attachment in $B$, and private
vertices carry the copy label.  The resulting $k$ petals are pairwise
disjoint, so $E_x$ is the base of a copy of $C^r_{1,k}$.
\end{proof}

Let $\beta_k(G)$ be the minimum size of a $k$-safe set, when one exists.  Theorem~\ref{thm:safe} and the affine array construction imply
\begin{equation}\label{eq:beta-transfer}
 \liminf_{n\to\infty}
 \frac{\exlin_r(n,C^r_{1,k})}{n}
 \ge \frac{|E(G)|}{\beta_k(G)}.
\end{equation}
Thus nonintersecting templates enlarge the search space.  Within intersecting templates the safe-block relaxation has the same ceiling.

\begin{proposition}[Intersecting safe blocks]\label{prop:safe-ceiling}
Let $G$ be a finite linear intersecting $r$-uniform hypergraph and let $B$ be $k$-safe.  Put $h=r-k+2$.  Then
\begin{equation}\label{eq:safe-ceiling}
 \frac{|E(G)|}{|B|}\le\frac rh.
\end{equation}
\end{proposition}

\begin{proof}
First suppose $G$ is not a star.  Then $\Delta(G)\le r$ by the argument in Theorem~\ref{thm:ceiling}.  Since $G$ is intersecting, $\mu_B(E)\le1$ for every edge $E$: two distinct edges in a family counted by $\mu_B(E)$ would have to be disjoint, contradicting intersection.

Safety implies $|B\cap E|\ge h-1$ for every edge.  If $|B\cap E|=h-1$, then $|E\setminus B|=k-1$, so \eqref{eq:safe} forces $\mu_B(E)=0$.  Every selected vertex on such an edge has degree one; otherwise another edge through that vertex would witness $\mu_B(E)\ge1$.

Let $m=|E(G)|$ and let $t$ be the number of edges with exactly $h-1$ selected vertices.  There are at least $(h-1)t$ distinct selected degree-one vertices.  If
\[
 I(B)=\sum_{E\in E(G)}|E\cap B|=\sum_{v\in B}d_G(v),
\]
then
\[
 I(B)\ge hm-t
\]
and, because the $(h-1)t$ degree-one vertices fall short of degree $r$ by $r-1$ each,
\[
 I(B)\le r|B|-(r-1)(h-1)t.
\]
Hence
\[
 r|B|\ge hm+\bigl((r-1)(h-1)-1\bigr)t\ge hm,
\]
which gives \eqref{eq:safe-ceiling}.

If $G$ is a star with $m$ edges, omit the center and select $h-1$ leaves from every edge; this is safe and has size $(h-1)m$.  Conversely, if a safe set omits the center, safety forces at least $h-1$ selected leaves on each edge.  If it contains the center and $m\ge2$, then $\mu_B(E)=1$ for every edge, so safety again forces at least $h-1$ selected leaves on each edge; when $m=1$, safety forces at least $h-1$ selected vertices in total.  Thus every safe set has size at least $(h-1)m$, and
\[
 \frac{m}{|B|}\le\frac1{h-1}\le\frac rh.
\]
\end{proof}

Proposition~\ref{prop:safe-ceiling} identifies a concrete frontier: any safe-block construction that beats $r/h$ must use a genuinely nonintersecting template.

\section{Open problems}\label{sec:conclusion}

The results above determine the optimal efficiency of the intersecting-block construction up to the remaining finite-geometric existence questions, and asymptotically whenever the crown size is linear in the uniformity.  The unrestricted coefficient $\clin_r(C^r_{1,k})$ may be larger; the safe-block formulation gives a concrete way to test whether nonintersecting templates improve the leading constant.  Four natural problems remain.

\begin{problem}[Odd-order additive gap]\label{prob:odd-gap}
Let $q$ be an odd prime power and $2\le h<q$.  Determine
\[
 \rho_{q+1,q-h+3}.
\]
In particular, can a non-projective intersecting block attain $(q+1)/h$ even though $\PG(2,q)$ has no nontrivial maximal $h$-arc, or is the truncated-plane value $q/h$ optimal in some odd-order cases?
\end{problem}

\begin{problem}[Stability of efficient blocks]\label{prob:stability}
Classify pairs $(G,B)$ for which
\[
 r|B|-h|E(G)|=o(|E(G)|).
\]
The defect identity forces most blocker incidences to lie on degree-$r$ vertices and most edges to be covered almost exactly $h$ times.  Determine whether linearity and intersection force a geometric model after deleting a lower-order set of incidences.
\end{problem}

\begin{problem}[Beyond the intersecting ceiling]\label{prob:beyond}
Find a nonintersecting linear $r$-uniform template $G$ satisfying
\[
 \frac{|E(G)|}{\beta_k(G)}>\frac{r}{r-k+2},
\]
or prove that no such example exists in a broad natural class of templates.
\end{problem}

\begin{problem}[Optimal constrained-port cost]\label{prob:ports}
For a graph $J$ and $N$ rows, determine the minimum possible value of
\[
 \sum_{u\in V(J)}|Y_u|
\]
over all $J$-constrained port arrays.  The complete graph has the familiar strength-two packing-array scale $\Theta(|V(J)|\sqrt N)$ at affine-plane orders, while the empty graph has constant cost.  Determine how intermediate graph structure controls the optimal overhead.
\end{problem}

\end{document}